\documentclass{amsart}
\usepackage{latexsym,amssymb,amsmath,amsthm,amscd,graphicx}

\makeatletter
\@namedef{subjclassname@2010}{%
  \textup{2010} Mathematics Subject Classification}
\makeatother

\numberwithin{equation}{section}
\newtheorem{theorem}{Theorem}[section]
\newtheorem{lemma}[theorem]{Lemma}

\newtheorem{proposition}[theorem]{Proposition}

\theoremstyle{definition}

\newtheorem{remark}[theorem]{Remark}

\theoremstyle{remark}

\newcommand{\cD}{{\mathcal D}}

\newcommand{\cF}{{\mathcal F}}
\newcommand{\cG}{{\mathcal G}}

\newcommand{\cW}{{\mathcal W}}

\newcommand{\R}{{\mathbb R}}

\newcommand{\Z}{{\mathbb Z}}

\def\al{\alpha}

\def\sg{\sigma}
\def\om{\omega}

\def\6{\partial}
\def\8{\infty}

\def\l{\left}
\def\r{\right}

\def\ol{\overline}

\begin{document}
\title
[A characterization of two weight trace inequalities]
{A characterization of two weight trace inequalities for positive dyadic operators in the upper triangle case}
\author[H.~Tanaka]{Hitoshi Tanaka}
\address{Graduate School of Mathematical Sciences, The University of Tokyo, Tokyo, 153-8914, Japan}
\email{htanaka@ms.u-tokyo.ac.jp}
\thanks{
The author is supported by 
the FMSP program at Graduate School of Mathematical Sciences, the University of Tokyo, 
and Grant-in-Aid for Scientific Research (C) (No.~23540187), 
the Japan Society for the Promotion of Science. 
}
\date{}

\begin{abstract}
Two weight trace inequalities for positive dyadic operators 
are characterized in terms of discrete Wolff's potentials 
in the upper triangle case 
$1<q<p<\8$.
\end{abstract}

\maketitle

\section{Introduction}\label{sec1}
The purpose of this paper is to establish the two weight $T1$ theorem for 
positive dyadic operators in the upper triangle case 
$1<q<p<\8$.
We first fix some notations. 
We will denote $\cD$ by the family of all dyadic cubes 
$Q=2^{-i}(k+[0,1)^n)$, 
$i\in\Z,\,k\in\Z^n$. 
Let $\sg$ and $\om$ be nonnegative Radon measures on $\R^n$ and 
let $K:\cD\to[0,\8)$ be a map. 
For an $f\in L_{{\rm loc}}^1(d\sg)$ 
the positive dyadic operator $T_{K}[fd\sg]$ is defined by 
$$
T_{K}[fd\sg](x)
:=
\sum_{Q\in\cD}
K(Q)\int_{Q}f\,d\sg1_{Q}(x)
\quad x\in\R^n.
$$
We will denote by $\ol{K}_{\sg}(Q)(x)$ the function
$$
\ol{K}_{\sg}(Q)(x)
:=
\frac1{\sg(Q)}\sum_{Q'\subset Q}
K(Q')\sg(Q')1_{Q'}(x),
\quad x\in Q\in\cD,
$$
and $\ol{K}_{\sg}(Q)(x)=0$ when $\sg(Q)=0$. 
For $s>1$ discrete Wolff's potential of $\om$ 
$\cW^s_{K,\sg}[\om](x)$ 
is defined by 
$$
\cW^s_{K,\sg}[\om](x)
:=
\sum_{Q\in\cD}
K(Q)\sg(Q)
\l(\int_{Q}\ol{K}_{\sg}(Q)(y)\,d\om(y)\r)^{s-1}
1_{Q}(x),
\quad x\in\R^n.
$$
The pair $(K,\sg)$ is said to be satisfy 
the dyadic logarithmic bounded oscillation (DLBO) condition, 
if they fulfill 
$$
\sup_{x\in Q}\ol{K}_{\sg}(Q)(x)
\le A
\inf_{x\in Q}\ol{K}_{\sg}(Q)(x),
$$
where the constant $A$ does not depend on $Q\in\cD$. 
For each $1<p<\infty$, 
$p'$ will denote the dual exponent of $p$, 
i.e., $p'=\frac{p}{p-1}$. 

In their significant paper \cite{CaOrVe2}, 
Cascante, Ortega and Verbitsky established the following: 

\begin{proposition}[{\rm\cite[Theorem A]{CaOrVe2}}]\label{prp1.1}
Let $0<q<p<\8$ and $1<p<\8$. 
Suppose that the pair $(K,\sg)$ satisfy the {\rm DLBO} condition. 
Then two weight trace inequality 
\begin{equation}\label{1.1}
\|T_{K}[fd\sg]\|_{L^q(d\om)}
\le C_1
\|f\|_{L^p(d\sg)}
\end{equation}
holds if and only if 
$$
\|\cW^{p'}_{K,\sg}[\om]^{1/p'}\|_{L^r(d\om)}
\le C_2<\8,
\text{ where }
\frac1q=\frac1r+\frac1p.
$$
Moreover, 
the least possible $C_1$ and $C_2$ are equivalent.
\end{proposition}

In his elegant paper \cite{Tr} 
Sergei Treil gives a simple proof of the following 
two weight $T1$ theorem for positive dyadic operators in the lower triangle case. 

\begin{proposition}[{\rm\cite[Theorem 2.1]{Tr}}]\label{prp1.2}
Let $1<p\le q<\8$. Then 
two weight trace inequality {\rm\eqref{1.1}} 
holds if and only if 
$$
\begin{cases}
\sup_{Q\in\cD}
\frac1{\sg(Q)^{1/p}}
\l(\int_{Q}\l(\sum_{Q'\subset Q}K(Q')\om(Q')1_{Q'}\r)^q\,d\om\r)^{1/q}
\le C_2<\8,
\\
\sup_{Q\in\cD}
\frac1{\om(Q)^{1/q'}}
\l(\int_{Q}\l(\sum_{Q'\subset Q}K(Q')\sg(Q')1_{Q'}\r)^{p'}\,d\sg\r)^{1/p'}
\le C_2<\8.
\end{cases}
$$
Moreover,
the least possible $C_1$ and $C_2$ are equivalent.
\end{proposition}
Proposition \ref{prp1.2} was first proved for $p=2$ in \cite{NaTrVo} 
by the Bellman function method. 
Later in \cite{LaSaUr} this was proved in full generality 
$1<p\le q<\8$. 
The checking condition in Proposition \ref{prp1.2} is called 
\lq\lq Sawyer type checking condition", 
since this was first introduced by Eric T. Sawyer 
in \cite{Sa1,Sa2}. 

In his excellent survey of the $A_2$ theorem \cite{Hy} 
Tuomas P. Hyt\"{o}nen introduces 
another proof of Proposition \ref{prp1.2}, 
which uses the \lq\lq parallel corona" decomposition 
from the recent work of 
Lacey, Sawyer, Shen and Uriarte-Tuero 
\cite{LaSaShUr} 
on the two weight boundedness of the Hilbert transform. 

Following Hyt\"{o}nen's arguments and 
applying a basic lemma due to \cite{CaOrVe1}, 
we shall establish the following 
two weight $T1$ theorem for positive dyadic operators in the upper triangle case. 

\begin{theorem}\label{thm1.3}
Let $1<q<p<\8$. Then 
two weight trace inequality {\rm\eqref{1.1}} 
holds if and only if 
$$
\begin{cases}
\|\cW^q_{K,\om}[\sg]^{1/q}\|_{L^r(d\sg)}
\le C_2<\8,
\\
\|\cW^{p'}_{K,\sg}[\om]^{1/p'}\|_{L^r(d\om)}
\le C_2<\8,
\\
\qquad\text{ where }
\frac1q=\frac1r+\frac1p.
\end{cases}
$$
Moreover,
the least possible $C_1$ and $C_2$ are equivalent.
\end{theorem}

\begin{remark}\label{rem1.4}
The {\rm DLBO} condition is essential and quite useful. 
In \cite{TaTe}, 
we develop a theory of weights for positive operators in a filtered measure space based upon this condition. 
\end{remark}

The letter $C$ will be used for constants
that may change from one occurrence to another.
Constants with subscripts, such as $C_1$, $C_2$, do not change
in different occurrences.

\section{Proof of Theorem \ref{thm1.3}}\label{sec2}
In what follows 
we shall prove Theorem \ref{thm1.3}. 
We need a basic lemma \cite[Theorem 2.1]{CaOrVe1}. 
For the sake of completeness, we will give the proof and 
will also check the constants. 

\begin{lemma}\label{lem2.1}
Let $\sg$ be a Radon measure on $\R^n$. 
Let $1<s<\8$ and 
$\{\al_{Q}\}_{Q\in\cD}\subset[0,\8)$.
Define, for $Q_0\in\cD$, 
\begin{alignat*}{2}
A_1
:=
\int_{Q_0}
\l(\sum_{Q\subset Q_0}\frac{\al_{Q}}{\sg(Q)}1_{Q}\r)^s
\,d\sg,
\\
A_2
:=
\sum_{Q\subset Q_0}
\al_{Q}\l(\frac1{\sg(Q)}\sum_{Q'\subset Q}\al_{Q'}\r)^{s-1},
\\
A_3
:=
\int_{Q_0}
\sup_{x\in Q\subset Q_0}
\l(\frac1{\sg(Q)}\sum_{Q'\subset Q}\al_{Q'}\r)^s
\,d\sg(x).
\end{alignat*}
Then 
$$
A_1\le c(s)A_2,\quad
A_2\le c(s)^{\frac1{s-1}}A_3
\text{ and }
A_3\le (s')^sA_1.
$$
Here, 
$$
c(s)
:=
\begin{cases}
s,\quad 1<s\le 2, 
\\
\l(s(s-1)\cdots(s-k)\r)^{\frac{s-1}{s-k-1}},
\quad 2<s<\8,
\end{cases}
$$
where $k=\lceil s-2 \rceil$ is 
the smallest integer greater than $s-2$. 
\end{lemma}

\begin{proof}
By a standard limiting argument, 
we may assume without loss of generality that 
there is only a finite number of $\al_{Q}\ne 0$. 

\noindent{\bf (i)}\ \ 
We prove $A_1\le c(s)A_2$. 
We use an elementary inequality
\begin{equation}\label{2.1}
\l(\sum_ia_i\r)^s
\le s
\sum_ia_i\l(\sum_{j\ge i}a_j\r)^{s-1},
\end{equation}
where $\{a_i\}_{i\in\Z}$ is a sequence of summable nonnegative reals. 
First, we verify the simple case 
$1<s\le 2$.
It follows from \eqref{2.1} that
\begin{alignat*}{2}
A_1
&=
\int_{Q_0}
\l(\sum_{Q\subset Q_0}\frac{\al_{Q}}{\sg(Q)}1_{Q}\r)^s
\,d\sg
\\ &\le s
\sum_{Q\subset Q_0}
\frac{\al_{Q}}{\sg(Q)}
\int_{Q}
\l(\sum_{Q'\subset Q}\frac{\al_{Q'}}{\sg(Q')}1_{Q'}\r)^{s-1}
\,d\sg
\\ &\le s
\sum_{Q\subset Q_0}
\al_{Q}
\l(\frac1{\sg(Q)}\int_{Q}
\l(\sum_{Q'\subset Q}\frac{\al_{Q'}}{\sg(Q')}1_{Q'}\r)
\,d\sg\r)^{s-1}
\\ &=s
\sum_{Q\subset Q_0}
\al_{Q}
\l(\frac1{\sg(Q)}\sum_{Q'\subset Q}\al_{Q'}\r)^{s-1}
=s
A_2,
\end{alignat*}
where we have used $s-1\le 1$ and 
H\"{o}lder's inequality. 
Next, we prove the case $s>2$. 
Let $k=\lceil s-2 \rceil$ be 
the smallest integer greater than $s-2$. 
Applying \eqref{2.1} $(k+1)$-times, we have 
\begin{alignat*}{2}
A_1
&=
s(s-1)\cdots(s-k)
\\ &\times
\sum_{P_k\subset\cdots\subset P_1\subset P_0\subset Q_0}
\frac{\al_{P_0}}{\sg(P_0)}
\frac{\al_{P_1}}{\sg(P_1)}
\ldots
\frac{\al_{P_k}}{\sg(P_k)}
\int_{P_k}
\l(\sum_{P\subset P_k}\frac{\al_{P}}{\sg(P)}1_{P}\r)^{s-k-1}
\,d\sg.
\end{alignat*}
Since we have $0<s-k-1\le 1$, 
\begin{alignat*}{2}
\lefteqn{
\frac1{\sg(P_k)}\int_{P_k}
\l(\sum_{P\subset P_k}\frac{\al_{P}}{\sg(P)}1_{P}\r)^{s-k-1}
\,d\sg
}\\ &\le
\l(\frac1{\sg(P_k)}\sum_{P\subset P_k}\al_{P}\r)^{s-k-1}.
\end{alignat*}
These yield 
\begin{alignat*}{2}
A_1
&\le
s(s-1)\cdots(s-k)
\\ &\times
\int_{Q_0}
\l(\sum_{Q\subset Q_0}\frac{\al_{Q}}{\sg(Q)}1_{Q}\r)^k
\l(
\sum_{Q\subset Q_0}
\frac{\al_{Q}}{\sg(Q)}
\l(\frac1{\sg(Q)}\sum_{Q'\subset Q}\al_{Q'}\r)^{s-k-1}
1_{Q}
\r)
\,d\sg.
\end{alignat*}
H\"{o}lder's inequality with exponent 
$\frac{k}{s-1}+\frac{s-k-1}{s-1}=1$
gives 
\begin{alignat*}{2}
\lefteqn{
\sum_{Q\subset Q_0}
\frac{\al_{Q}}{\sg(Q)}
\l(\frac1{\sg(Q)}\sum_{Q'\subset Q}\al_{Q'}\r)^{s-k-1}
1_{Q}
}\\ &\le
\l(\sum_{Q\subset Q_0}\frac{\al_{Q}}{\sg(Q)}1_{Q}\r)^{\frac{k}{s-1}}
\l(
\sum_{Q\subset Q_0}
\frac{\al_{Q}}{\sg(Q)}
\l(\frac1{\sg(Q)}\sum_{Q'\subset Q}\al_{Q'}\r)^{s-1}
1_{Q}
\r)^{\frac{s-k-1}{s-1}},
\end{alignat*}
and, hence, 
\begin{alignat*}{2}
A_1
&\le 
s(s-1)\cdots(s-k)
\\ &\times
\int_{Q_0}
\l(\sum_{Q\subset Q_0}\frac{\al_{Q}}{\sg(Q)}1_{Q}\r)^{\frac{ks}{s-1}}
\l(
\sum_{Q\subset Q_0}
\frac{\al_{Q}}{\sg(Q)}
\l(\frac1{\sg(Q)}\sum_{Q'\subset Q}\al_{Q'}\r)^{s-1}
1_{Q}
\r)^{\frac{s-k-1}{s-1}}
\,d\sg.
\end{alignat*}
H\"{o}lder's inequality with the same exponent gives 
$$
A_1
\le 
s(s-1)\cdots(s-k)
A_1^{\frac{k}{s-1}}
A_2^{\frac{s-k-1}{s-1}}.
$$
Thus, we obtain 
$A_1\le c(s)A_2$.

\noindent{\bf (ii)}\ \ 
We prove $A_2\le c(s)^{\frac1{s-1}}A_3$. 
It follows that 
\begin{alignat*}{2}
A_2
&=
\int_{Q_0}
\sum_{Q\subset Q_0}
\frac{\al_{Q}}{\sg(Q)}
\l(\frac1{\sg(Q)}\sum_{Q'\subset Q}\al_{Q'}\r)^{s-1}
1_{Q}
\,d\sg
\\ &\le
\int_{Q_0}
\l(\sum_{Q\subset Q_0}
\frac{\al_{Q}}{\sg(Q)}1_{Q}(x)\r)
\l(
\sup_{x\in Q\subset Q_0}
\frac1{\sg(Q)}\sum_{Q'\subset Q}\al_{Q'}
\r)^{s-1}
\,d\sg(x).
\end{alignat*}
H\"{o}lder's inequality gives 
$$
A_2
\le
A_1^{\frac1s}
A_3^{\frac1{s'}}.
$$
Since we have had $A_1\le c(s)A_2$, 
we obtain 
$A_2\le c(s)^{\frac1{s-1}}A_3$.

\noindent{\bf (iii)}\ \ 
We prove $A_3\le (s')^sA_1$. 
It follows that 
\begin{alignat*}{2}
A_3
&=
\int_{Q_0}
\sup_{x\in Q\subset Q_0}
\l(\frac1{\sg(Q)}\sum_{Q'\subset Q}\al_{Q'}\r)^s
\,d\sg(x)
\\ &\le
\int_{Q_0}
M_{\sg}\l[\sum_{Q\subset Q_0}\frac{\al_{Q}}{\sg(Q)}1_{Q}\r](x)^s
\,d\sg(x)
\\ &\le (s')^s
A_1,
\end{alignat*}
where 
$M_{\sg}$ is the dyadic Hardy-Littlewood maximal operator and 
we have used the $L^s(d\sg)$-boundedness of $M_{\sg}$. 
This completes the proof. 
\end{proof}

\noindent{\bf Proof of Theorem \ref{thm1.3} (Sufficiency):}\ \ 
We follow the arguments due to Hyt\"{o}nen in \cite{Hy}. 
Let $Q_0\in\cD$ be taken large enough and be fixed. 
We shall estimate the quantity 
\begin{equation}\label{2.2}
\sum_{Q\subset Q_0}
K(Q)\int_{Q}f\,d\sg\int_{Q}g\,d\om,
\end{equation}
where 
$f\in L^p(d\sg)$ and $g\in L^{q'}(d\om)$ 
are nonnegative and are supported in $Q_0$. 

We define the collections of principal cubes 
$\cF$ for the pair $(f,\sg)$ and 
$\cG$ for the pair $(g,\om)$. 
Namely, analogously for $\cG$, 
$$
\cF:=\bigcup_{k=0}^{\8}\cF_k,
$$
where 
$\cF_0:=\{Q_0\}$,
$$
\cF_{k+1}:=\bigcup_{F\in\cF_k}ch_{\cF}(F)
$$
and $ch_{\cF}(F)$ is defined by 
the set of all maximal dyadic cubes $Q\subset F$ such that 
$$
\frac1{\sg(Q)}\int_{Q}f\,d\sg
>
\frac2{\sg(F)}\int_{F}f\,d\sg.
$$
Observe that
\begin{alignat*}{2}
\lefteqn{
\sum_{F'\in ch_{\cF}(F)}\sg(F')
}\\ &\le
\l(\frac2{\sg(F)}\int_{F}f\,d\sg\r)^{-1}
\sum_{F'\in ch_{\cF}(F)}
\int_{F'}f\,d\sg
\\ &\le
\l(\frac2{\sg(F)}\int_{F}f\,d\sg\r)^{-1}
\int_{F}f\,d\sg
=
\frac{\sg(F)}{2},
\end{alignat*}
and, hence, 
\begin{equation}\label{2.3}
\sg(E_{\cF}(F))
:=
\sg\l(F\setminus\bigcup_{F'\in ch_{\cF}(F)}F'\r)
\ge
\frac{\sg(F)}{2},
\end{equation}
where the sets $E_{\cF}(F)$ are pairwise disjoint. 

We further define the stopping parents, 
for $Q\in\cD$, 
$$
\begin{cases}
\pi_{\cF}(Q)
:=
\min\{F\supset Q:\,F\in\cF\},
\\
\pi_{\cG}(Q)
:=
\min\{G\supset Q:\,G\in\cG\},
\\
\pi(Q)
:=
(\pi_{\cF}(Q),\pi_{\cG}(Q)).
\end{cases}
$$
Then we can rewrite the series in \eqref{2.2} as follows:
\begin{alignat*}{2}
\sum_{Q\subset Q_0}
&=
\sum_{\substack{F\in\cF, \\ G\in\cG}}
\sum_{\substack{Q: \\ \pi(Q)=(F,G)}}
\\ &\le
\sum_{F\in\cF}\sum_{G\subset F}
\sum_{\substack{Q: \\ \pi(Q)=(F,G)}}
+
\sum_{G\in\cG}\sum_{F\subset G}
\sum_{\substack{Q: \\ \pi(Q)=(F,G)}},
\end{alignat*}
where we have used the fact that 
if $P,Q\in\cD$ then 
$P\cap Q\in\{P,Q,\emptyset\}$. 
Since the proof can be done in completely symmetric way, 
we shall concentrate ourselves on the first case only.

It follows that, for $F\in\cF$, 
\begin{alignat*}{2}
\lefteqn{
\sum_{G\subset F}
\sum_{\substack{Q: \\ \pi(Q)=(F,G)}}
K(Q)\int_{Q}f\,d\sg\int_{Q}g\,d\om
}\\ &=
\sum_{G\subset F}
\sum_{\substack{Q: \\ \pi(Q)=(F,G)}}
K(Q)\sg(Q)
\l(\frac1{\sg(Q)}\int_{Q}f\,d\sg\r)
\int_{Q}g\,d\om
\\ &\le
\frac2{\sg(F)}\int_{F}f\,d\sg
\sum_{G\subset F}
\sum_{\substack{Q: \\ \pi(Q)=(F,G)}}
K(Q)\sg(Q)\int_{Q}g\,d\om.
\end{alignat*}
We need the two observations. 
Suppose that 
$\pi(Q)=(F,G)$ and $G\subset F$. 
If $F'\in ch_{\cF}(F)$ satisfies $F'\subset Q$, 
then by definition of $\pi_{\cF}$ we must have 
$$
\pi_{\cF}\l(\pi_{\cG}(F')\r)=F.
$$
By this observation we define 
$$
ch_{\cF}^*(F)
:=
\l\{F'\in ch_{\cF}(F):\,
\pi_{\cF}\l(\pi_{\cG}(F')\r)=F\r\}.
$$
We further observe that, 
when $F'\in ch_{\cF}^*(F)$, 
we can regard $g$ as a constant on $F'$ in the above integrals. 
By these observations we see that, 
by use of H\"{o}lder's inequality, 
\begin{alignat*}{2}
\lefteqn{
\sum_{G\subset F}
\sum_{\substack{Q: \\ \pi(Q)=(F,G)}}
K(Q)\sg(Q)\int_{Q}g\,d\om
}\\ &\le
\l(\int_{F}\l(\sum_{Q\subset F}K(Q)\sg(Q)1_{Q}\r)^q\,d\om\r)^{1/q}
\\ &\quad\times
\l(
\int_{E_{\cF}(F)}g^{q'}\,d\om
+
\sum_{F'\in ch_{\cF}^*(F)}
\l(\frac1{\om(F')}\int_{F'}g\,d\om\r)^{q'}\om(F')
\r)^{1/q'}
\\ &=:
\l(\int_{F}\l(\sum_{Q\subset F}K(Q)\sg(Q)1_{Q}\r)^q\,d\om\r)^{1/q}
\|g_{F}\|_{L^{q'}(d\om)}.
\end{alignat*}
Thus, we obtain 
\begin{alignat*}{2}
\lefteqn{
\sum_{F\in\cF}\sum_{G\subset F}
\sum_{\substack{Q: \\ \pi(Q)=(F,G)}}
K(Q)\int_{Q}f\,d\sg\int_{Q}g\,d\om
}\\ &\le
\sum_{F\in\cF}
\frac2{\sg(F)}\int_{F}f\,d\sg
\l(\int_{F}\l(\sum_{Q\subset F}K(Q)\sg(Q)1_{Q}\r)^q\,d\om\r)^{1/q}
\|g_{F}\|_{L^{q'}(d\om)}
\\ &\le 2
\l(
\sum_{F\in\cF}
\l(\frac1{\sg(F)}\int_{F}f\,d\sg\r)^p
\sg(F)
\r)^{1/p}
\\ &\quad\times
\l[
\sum_{F\in\cF}
\l\{\frac1{\sg(F)^{1/p}}\l(\int_{F}\l(\sum_{Q\subset F}K(Q)\sg(Q)1_{Q}\r)^q\,d\om\r)^{1/q}\r\}^{p'}
\|g_{F}\|_{L^{q'}(d\om)}^{p'}
\r]^{1/p'}
\\ =:2
I_1\times I_2.
\end{alignat*}
For $I_1$, using 
$\sg(F)\le 2\sg(E_{\cF}(F))$, 
$$
\frac1{\sg(F)}\int_{F}f\,d\sg
\le
\inf_{y\in F}M_{\sg}f(y)
$$
and the disjointness of the $E_{\cF}(F)$,
we have 
\begin{alignat*}{2}
I_1
&\le 2^{1/p}
\l(
\sum_{F\in\cF}
\int_{E_{\cF}(F)}(M_{\sg}f)^p\,d\sg
\r)^{1/p}
\\ &\le 2^{1/p}
\l(\int_{Q_0}(M_{\sg}f)^p\,d\sg\r)^{1/p}
\le 2^{1/p}p'
\|f\|_{L^p(d\sg)}.
\end{alignat*}
Recall that $\frac1q=\frac1r+\frac1p$ 
and let $\theta:=\frac{q'}{p'}$. Then 
we have $\theta>1$ and 
${\theta'}{p'}=r$. 
It follows from H\"{o}lder's inequality 
with exponent $\theta$ that 
\begin{alignat*}{2}
I_2
&\le
\l[
\sum_{F\in\cF}
\l\{\frac1{\sg(F)^{1/p}}\l(\int_{F}\l(\sum_{Q\subset F}K(Q)\sg(Q)1_{Q}\r)^q\,d\om\r)^{1/q}\r\}^r
\r]^{1/r}
\\ &\quad\times
\l(
\sum_{F\in\cF}
\|g_{F}\|_{L^{q'}(d\om)}^{q'}
\r)^{1/q'}
\\ =:
I_{21}\times I_{22}.
\end{alignat*}
It follows by applying Lemma \ref{lem2.1} that 
\begin{alignat*}{2}
\lefteqn{
\int_{F}\l(\sum_{Q\subset F}K(Q)\sg(Q)1_{Q}\r)^q\,d\om
}\\ &\le c(q)
\sum_{Q\subset F}
K(Q)\sg(Q)\om(Q)
\l(\frac1{\om(Q)}\sum_{Q'\subset Q}K(Q')\sg(Q')\om(Q')\r)^{q-1}
\\ &=c(q)
\int_{F}
\sum_{Q\subset F}
K(Q)\om(Q)
\l(\int_{Q}\ol{K}_{\om}(Q)(y)\,d\sg(y)\r)^{q-1}
1_{Q}
\,d\sg.
\end{alignat*}
This implies 
\begin{alignat*}{2}
\lefteqn{
\l\{\frac1{\sg(F)^{1/p}}\l(\int_{F}\l(\sum_{Q\subset F}K(Q)\sg(Q)1_{Q}\r)^q\,d\om\r)^{1/q}\r\}^r
}\\ &\le c(q)^{r/q}
\l(
\frac1{\sg(F)}\int_{F}
\sum_{Q\subset F}
K(Q)\om(Q)
\l(\int_{Q}\ol{K}_{\om}(Q)(y)\,d\sg(y)\r)^{q-1}
1_{Q}
\,d\sg
\r)^{r/q}
\sg(F)
\\ &\le 2c(q)^{r/q}
\int_{E_{\cF}(F)}
\l(M_{\sg}\cW^q_{K,\om}[\sg]\r)^{r/q}
\,d\sg,
\end{alignat*}
and, hence, 
$$
I_{21}
\le 2^{1/q}c(q)^{1/q}(r/q)'
\|\cW^q_{K,\om}[\sg]^{1/q}\|_{L^r(d\sg)}.
$$
It remains to estimate $I_{22}$. 
it follows that 
$$
I_{22}^{q'}
=
\sum_{F\in\cF}
\int_{E_{\cF}(F)}g^{q'}\,d\om
+
\sum_{F\in\cF}
\sum_{F'\in ch_{\cF}^*(F)}
\l(\frac1{\om(F')}\int_{F'}g\,d\om\r)^{q'}\om(F').
$$
By the pairwise disjointness of the set $E_{\cF}(F)$, 
it is immediate that
$$
\sum_{F\in\cF}
\int_{E_{\cF}(F)}g^{q'}\,d\om
\le
\|g\|_{L^{q'}(d\om)}^{q'}.
$$
For the remaining double sum, 
we use the definition of 
$ch_{\cF}^*(F)$ 
to reorganize:
\begin{alignat*}{2}
\lefteqn{
\sum_{F\in\cF}
\sum_{F'\in ch_{\cF}^*(F)}
\l(\frac1{\om(F')}\int_{F'}g\,d\om\r)^{q'}\om(F')
}\\ &=
\sum_{F\in\cF}
\sum_{\substack{G\in\cG: \\ \pi_{\cF}(G)=F}}
\sum_{\substack{F'\in ch_{\cF}(F): \\ \pi_{\cG}(F')=G}}
\l(\frac1{\om(F')}\int_{F'}g\,d\om\r)^{q'}\om(F')
\\ &\le
\sum_{F\in\cF}
\sum_{\substack{G\in\cG: \\ \pi_{\cF}(G)=F}}
\l(\frac2{\om(G)}\int_{G}g\,d\om\r)^{q'}\om(G)
\\ &\le
\sum_{G\in\cG}
\l(\frac2{\om(G)}\int_{G}g\,d\om\r)^{q'}\om(G)
\\ &\le 2\cdot2^{q'}
\|M_{\om}g\|_{L^{q'}(d\om)}^{q'}
\le 2\cdot2^{q'}q^{q'}
\|g\|_{L^{q'}(d\om)}^{q'}.
\end{alignat*}
All together, we obtain 
$$
\sum_{Q\subset Q_0}
K(Q)\int_{Q}f\,d\sg\int_{Q}g\,d\om
\le C
\|\cW^q_{K,\om}[\sg]^{1/q}\|_{L^r(d\sg)}
\|f\|_{L^p(d\sg)}
\|g\|_{L^{Q'}(d\om)}.
$$
This yields the sufficiency of Theorem \ref{thm1.3}. 

\noindent{\bf Proof of Theorem \ref{thm1.3} (Necessity):}\ \ 
This fact was verified in \cite[Theorem B (i)]{CaOrVe1}. 
But, for reader's convenience the full proof is given here.
We assume that the trace inequality \eqref{1.1} holds. 
Then, by Lemma \ref{lem2.1}, there holds 
\begin{equation}\label{2.4}
\sum_{Q\in\cD}
K(Q)\om(Q)\int_{Q}f\,d\sg
\l(
\frac1{\om(Q)}\sum_{Q'\subset Q}
K(Q')\om(Q')\int_{Q'}f\,d\sg
\r)^{q-1}
\le C\,C_1^q
\|f\|_{L^p(d\sg)}^q,
\end{equation}
where $f\in L^p(d\sg)$ is nonnegative. 
For $g\ge 0$ we have 
\begin{alignat*}{2}
\lefteqn{
\int_{\R^n}g(x)\cW^q_{K,\om}[\sg](x)\,d\sg(x)
}\\ &=
\sum_{Q\in\cD}
K(Q)\om(Q)\int_{Q}g\,d\sg
\l(
\frac1{\om(Q)}\sum_{Q'\subset Q}
K(Q')\om(Q')\sg(Q')
\r)^{q-1}
\\ &=
\sum_{Q\in\cD}
K(Q)\om(Q)\sg(Q)
\l(\frac{\int_{Q}g\,d\sg}{\sg(Q)}\r)^{1/q}
\l(
\frac1{\om(Q)}
\l(\frac{\int_{Q}g\,d\sg}{\sg(Q)}\r)^{1/q}
\sum_{Q'\subset Q}
K(Q')\om(Q')\sg(Q')
\r)^{q-1}
\\ &\le
\sum_{Q\in\cD}
K(Q)\om(Q)\int_{Q}(M_{\sg}g)^{1/q}\,d\sg
\l(
\frac1{\om(Q)}\sum_{Q'\subset Q}
K(Q')\om(Q')\int_{Q'}(M_{\sg}g)^{1/q}\,d\sg
\r)^{q-1}
\\ &\le C\,C_1^q
\|(M_{\sg}g)^{1/q}\|_{L^p(d\sg)}^q
\\ &\le C\,C_1^q
\|g\|_{L^{p/q}(d\sg)},
\end{alignat*}
where we have used \eqref{2.4} and 
the $L^{p/q}(d\sg)$-boundedness of $M_{\sg}$. 
This implies by duality 
$$
\|\cW^q_{K,\om}[\sg]^{1/q}\|_{L^r(d\sg)}
\le C\,C_1<\8.
$$
To verify 
$$
\|\cW^{p'}_{K,\sg}[\om]^{1/p'}\|_{L^r(d\om)}
\le C\,C_1<\8,
$$
we merely use the dual inequality of \eqref{1.1}.


\begin{thebibliography}{999}

\bibitem{CaOrVe1}
Cascante~C., Ortega~J. and Verbitsky~I.,
\emph{Nonlinear potentials and two weight trace inequalities for general dyadic and radial kernels},
Indiana Univ. Math. J., \textbf{53} (2004), no. 3, 845--882.

\bibitem{CaOrVe2} \bysame,
\emph{On $L^p$-$L^q$ trace inequalities}, 
J. London Math. Soc. (2), \textbf{74} (2006), no. 2, 497--511. 

\bibitem{Hy} Hyt\"{o}nen~T., 
\emph{The $A_2$ theorem: Remarks and complements},
arXiv:1212.3840 (2012).

\bibitem{LaSaUr} 
Lacey~M., Sawyer~E. and Uriarte-Tuero~I., 
\emph{Two weight inequalities for discrete positive operators},
arXiv:0911.3437 (2009).

\bibitem{LaSaShUr} 
Lacey~M., Sawyer~E., Shen~C.-Y. and Uriarte-Tuero~I., 
\emph{Two weight inequality for the Hilbert transform: A real variable characterization},
arXiv:1201.4319 (2012).

\bibitem{NaTrVo} 
Nazarov~F., Treil~S. and Volberg~A.,
\emph{The Bellman functions and two-weight inequalities for Haar multipliers}, 
J. of Amer. Math. Soc., \textbf{12} (1999), no. 4, 909--928.

\bibitem{Sa1} Sawyer~E., 
\emph{A characterization of a two-weight norm inequality for maximal operators}, 
Studia Math., \textbf{75} (1982), no. 1, 1--11.

\bibitem{Sa2} \bysame,
\emph{A characterization of two weight norm inequalities for fractional and Poisson integrals},
Trans. Amer. Math. Soc., \textbf{308} (1988), no. 2, 533--545.

\bibitem{TaTe} Tanaka~H. and Terasawa~Y.,
\emph{Positive operators and maximal operators in a filtered measure space},
J. Funct. Anal., {\bf 264}(2013), no.4, 920--946. 

\bibitem{Tr} Treil~S.,
\emph{A remark on two weight estimates for positive dyadic operators}, 
arXiv:1201.1455 (2012).

\end{thebibliography}
\end{document}